\documentclass{article}
\usepackage{graphicx} 
\usepackage[utf8]{inputenc}
\usepackage[T1]{fontenc}
\usepackage[english]{babel}
\usepackage{geometry}

\usepackage{authblk}

\usepackage{amsfonts}
\usepackage{amsmath}
\usepackage{amssymb}
\usepackage{amstext}
\usepackage{amsthm}

\usepackage[shortlabels]{enumitem}
\usepackage{float}
\usepackage{graphicx}
\usepackage{hyperref}
\usepackage{listings}
\usepackage{mathrsfs}
\usepackage{mathtools}
\usepackage{multicol}
\usepackage{wasysym}
\usepackage{subcaption}
\usepackage[capitalize]{cleveref}

\usepackage{stmaryrd}

\usepackage{thm-restate}

\usepackage{silence}
\usepackage{tikz}
\usepackage{tikz-cd}
\usetikzlibrary{arrows, backgrounds, calc, chains, decorations, patterns, positioning, shapes}

\usetikzlibrary{snakes}

\usepackage{comment}
\specialcomment{idea}{\begingroup\color{blue}\itshape}{\endgroup}

\newtheorem{thm}{Theorem}

\newtheorem{cor}[thm]{Corollary}

\newtheorem{lemma}[thm]{Lemma}
\newtheorem{prop}[thm]{Proposition}

\newtheorem{claim}{Claim}[thm]
\newenvironment{proofclaim}[1][{\it Proof of claim. \hspace{0.066cm}}]%
	{\noindent {}{#1}{}}{ \strut\hfill $\lozenge$\vspace{2ex}}
\theoremstyle{definition}
\newtheorem*{defi*}{Definition}
\newtheorem*{thm*}{Theorem}
\theoremstyle{remark}

\newcommand{\CH}[1]{{\color{blue}{\bf CH:} #1}}
\newcommand{\ER}[1]{{\color{purple}{\bf ER:} #1}}
\newcommand{\PB}[1]{{\color{orange}{\bf PB:} #1}}

\newcommand{\NN}{\mathbb{N}}
\newcommand{\ZZ}{\mathbb{Z}}

\newcommand{\cleq}{\preccurlyeq}

\usepackage{babelbib}

\usepackage[capitalize]{cleveref}

\title{Path eccentricity of $k$-AT-free graphs and application on graphs with the consecutive ones property}
\date{v1 March 2024, v2 January 2025}

\author[1]{Paul Bastide}
\author[2]{Claire Hilaire\thanks{Partially supported by Slovenian Research and Innovation Agency (research project J1-4008 and N1-0370)}}
\author[3]{Eileen Robinson}

\affil[1]{LaBRI, Université de Bordeaux, France.} 
\affil[2]{Famnit, University of Primorska, Slovenia.} 
\affil[3]{Université libre de Bruxelles, Belgium.} 

\begin{document}

\maketitle

\begin{abstract}
The central path problem is a variation on the single facility location problem. The aim is to find, in a given connected graph $G$, a path $P$ minimizing its eccentricity, which is the maximal distance from $P$ to any vertex of the graph $G$. 
The {\em path eccentricity} of $G$ is the minimal eccentricity achievable over all paths in $G$.
In this article we consider the class of the $k$-AT-free graphs. They are graphs in which any set of three vertices contains a pair for which every path between them uses at least one vertex of the closed neighborhood at distance $k$ of the third. We prove that they have path eccentricity bounded by $k$.

Moreover, we contribute to answer a question of G\'omez and Guti\'errez, whether there was a relation between path eccentricity and the consecutive ones property. The latter is the property for a binary matrix to admit a permutation of the rows placing the 1's consecutively on the columns. It was already known that graphs whose adjacency matrices have the consecutive ones property have path eccentricity at most 1, and that the same remains true when the augmented adjacency matrices (with ones on the diagonal) has the consecutive ones property. We generalize these results as follow. We study graphs whose adjacency matrices can be made to satisfy the consecutive ones property after changing some values on the diagonal, and show that those graphs have path eccentricity at most 2, by showing that they are 2-AT-free.

\end{abstract}

\section{Introduction} 
The well-studied \emph{single facility location problem} consists in finding a location for a facility minimizing its distance to the communities it serves. For example, this could consists in finding the best place to build a fire station that has to reach any house very fast and therefore be as central as possible. Building a hospital, a police station and many other situations can raise similar concerns. This problems has been extensively studied and extended in many ways over the years (see \cite{SurveyOwenDaskin} for an in-depth survey) and has many practical applications in Geography, Economics and Computer Sciences, to only name a few. 

In this paper, we focus on the variation of this problem called the \emph{central path problem}, introduced independently by S.M. Hedetniemi, Cockayne and S.T. Hedetniemi \cite{linearAlgoPC_Cockayne} and by Slater \cite{LocatingCP}. Instead of a facility being a building, we consider railroads, transit routes, highways or pipelines for example, that have to be as close as possible to all the elements they have to serve.

To model this problem, one can use a graph $G$ where the set of vertices is the collection of communities, the set of edges is the set of roads connecting them, and the distance between two vertices is the smallest sequence of edges needed to go from one to the other.
The facility we want to find is modeled by a path in our graph, and we want it to be as close of the rest of the vertices of the graph as possible.

The notion capturing how far a path $P$ is to the rest of the graph is its eccentricity, which is the maximal distance between a vertex $u$ of the graph and the closest vertex $v$ of $P$.
It yields a invariant for graphs called the \emph{path eccentricity} of $G$, noted $pe(G)$, defined as the minimum of the eccentricity over all paths of $G$ (see formal definitions in Section~\ref{sec:preli}).

Both in \cite{linearAlgoPC_Cockayne} and \cite{LocatingCP}, the authors focus on finding a path with minimal eccentricity in trees. They also study algorithms computing such a path. They show that the class of trees has unbounded path eccentricity. 

Later on, Corneil,  Olariu and Stewart \cite{ATfree} considered the class of graphs with no asteroidal triples in depth and found a bound for their path eccentricity\footnote{In the literature, a path $P$ with eccentricity at most $k$ is also called a $k$-dominating path. Corneil et al.~\cite{ATfree} actually gave a stronger result as they showed that every AT-free graph has a 1-dominating path that has a minimum number of vertices over all the connected 1-dominating subgraphs.}. 
An \emph{asteroidal triple} is a set of three vertices such that each pair of vertices is joined by a path that avoids the closed neighborhood of the third.
A graph $G$ is said to be \emph{asteroidal triple-free}, noted \emph{AT-free}, if it does not contain any asteroidal triple.
\begin{thm}[\cite{ATfree}]\label{th:ATfreeecc1}
    If $G$ is an AT-free graph, then $pe(G) \leq 1$.
\end{thm}

More recently, G\'omez and Guti\'errez~\cite{PathEcc} studied the path eccentricity of convex and biconvex graphs. 
A bipartite graph $G=(X \cup Y, E)$ is said to be \emph{X-convex} if there is a total ordering of $X$ such that the neighborhood of any vertex $y\in Y$ is consecutive in the given ordering. This graph is said to be \emph{biconvex} if it is both $X$- and $Y$-convex.
They proved the following two upper bounds and showed that they are both tight.

\begin{thm}[\cite{PathEcc}]\label{thm:GGconvex}
    Let $G=(X\cup Y,E)$ be a bipartite graph. If $G$ is $X$-convex then $pe(G) \leq 2$. If $G$ is also $Y$-convex (i.e. biconvex), then $pe(G) \leq 1$. 
\end{thm}

G\'omez and Guti\'erez observed in their paper that the adjacency matrix, or a variation of it, of the biconvex graphs and a subclass of AT-free graphs have the \emph{consecutive ones property} (for columns), shortened by \emph{C1P}, meaning that there exists a permutation of its rows that places the 1s consecutively in every column. 
They asked the following natural question: is there a deeper connection between the C1P and the path eccentricity? 
The motivation for this question is reinforced by a tour of the literature. Firstly, Fulkerson and Gross \cite{IncidenceMA} showed the equivalence between graphs whose dominant clique vs. vertex incidence matrix has the C1P and the interval graphs. A graph is an \emph{interval graph} if it can be represented in such a way that each vertex corresponds to a closed interval on the real line, and two vertices are adjacent if and only if their corresponding intervals intersect. 
They insist on the importance of those graphs in genetic theory and provide an algorithm to determine if a given matrix has the C1P in $O(n^2)$ (where $n$ is the number of rows). 
Moreover Gardi showed that a graph is a unit interval graph if and only if its augmented adjacency matrix has the C1P~\cite[Thm 1]{GardiInterval}. A graph is a \emph{unit interval graph} if it has an interval representation where all the intervals have unit length. 
Interval graphs form a subclass of AT-free graphs so they have path eccentricity at most 1 by Theorem~\ref{th:ATfreeecc1}. Secondly, Chen showed that a graph is biconvex if and only if its adjacency matrix has the C1P~\cite[Thm 6]{ChenBiconv}. Therefore, by Theorem~\ref{thm:GGconvex}, it has path eccentricity bounded by 1 too. 
We contribute to extend the understanding of this connection and prove that a generalization of the natural definition of the C1P on graphs implies bounded path eccentricity.

\subsection*{Our results} 

\begin{figure}
    \centering
    \includegraphics[scale=0.4]{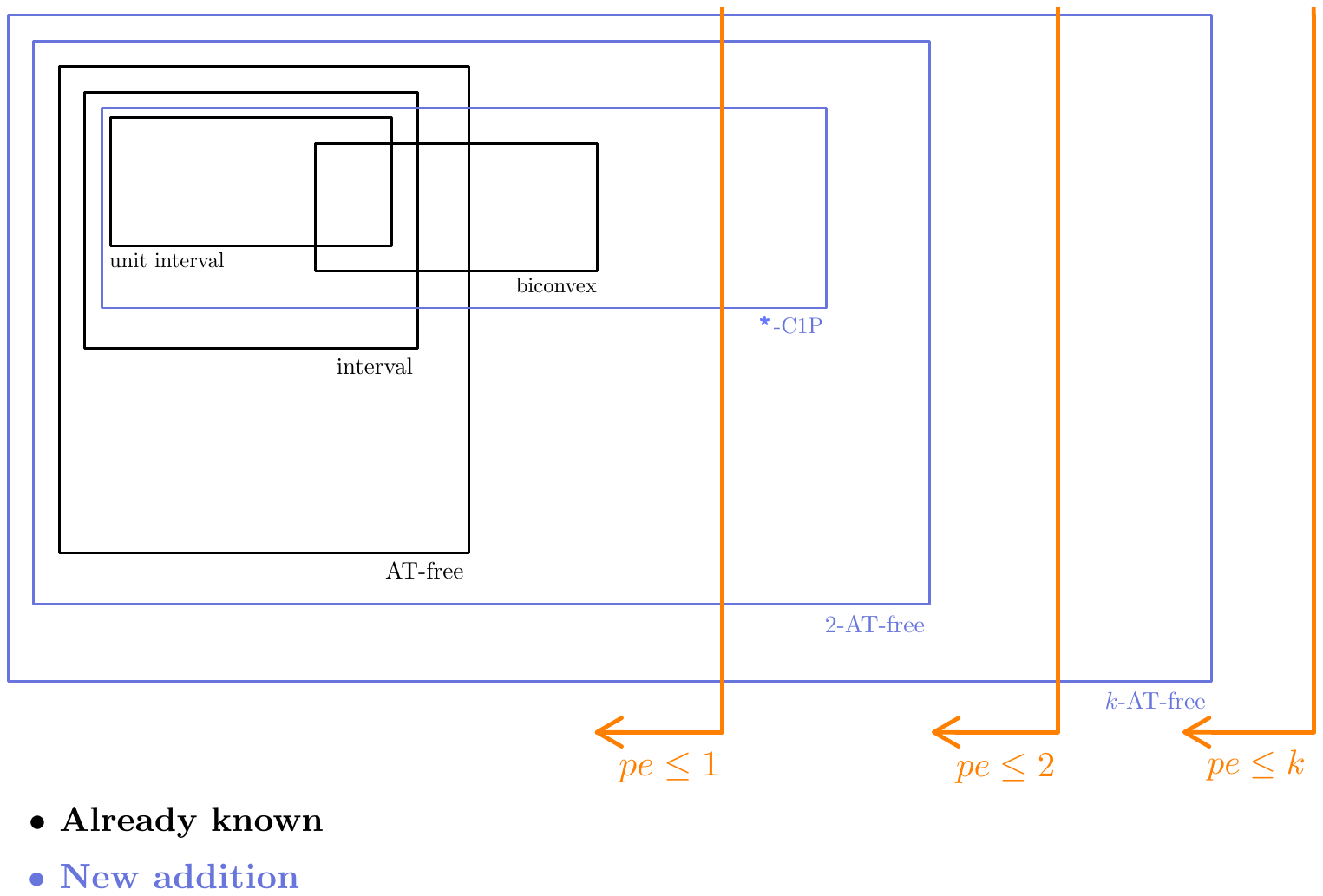}
    \caption{Summary of the path eccentricity of some classes of graphs.} 
    \label{fig:pe_world}
\end{figure}

We generalize 
\cref{th:ATfreeecc1} to \emph{$k$-asteroidal triple-free} graphs (or \emph{$k$-AT-free} graphs for short), a generalization of AT-free graphs introduced by Machado and de Figueiredo~\cite{kATfreeMachadoFigueiredo}. They are defined as graphs without \emph{$k$-asteroidal triples}, which are sets of three vertices such that each pair of vertices is joined by a path that avoids the neighborhood at distance $k$ of the third. Thus,  AT-free graphs correspond to $1$-AT-free graphs. We prove the following theorem.
\begin{restatable}{thm}{kATfree}
  \label{thm:kATfree_pek}
    For every $k\geq 1$, if a graph $G$ is $k$-AT-free then $pe(G)\leq k$.
\end{restatable}

We also address G\'omez and Guti\'erez's question about the relation between the C1P and path eccentricity. 

In order to generalize the classes of graphs whose adjacency or augmented adjacency matrix have the C1P,
we introduce the notion of graphs having the \emph{partially augmented consecutive ones property}, noted \emph{*-C1P}. 
A graph satisfies this property if a specific variant of its adjacency matrix has the C1P. 
Formally, a $n$-vertex graph $G$ has the *-C1P if there exists an adjacency matrix $(A_{i,j})_{(i,j) \in [n]^2}$ of $G$ and there exists $(A^*_{i,j})_{(i,j) \in [n]^2}$ with $A^*_{i,j} = A_{i,j}$ for all $i \neq j$ and $A^*_{i,i} \in \{0,1\}$ for all $i$, such that $(A^*_{i,j})_{(i,j) \in [n]^2}$ has the C1P.
This can be rephrased as the possibility of ordering the vertices in such a way that either the open or the closed neighborhood of each vertex is consecutive in the ordering (see Section~\ref{sec:preli}).


 We are to prove the following structural statement:

\begin{restatable}{thm}{twoATfree}\label{thm:mC1P_SAT}
    If a graph $G$ has the *-C1P, then it has no 2-AT.
\end{restatable}

Combining this result with Theorem~\ref{thm:kATfree_pek}, we deduce the following bound on the path eccentricity of graphs satisfying the *-C1P.
\begin{cor}
     If a graph $G$ has the *-C1P then $pe(G)\leq 2$. 
\end{cor}

These results widen the knowledge of the path eccentricity of graphs as summarized in Figure~\ref{fig:pe_world}.
Observe that the *-C1P is a more restrictive property than being 2-AT-free. For example a cycle of 5 vertices has no 2-AT but does not have the *-C1P. Therefore, the upper bound on the path eccentricity for graphs with the *-C1P is not necessarly tight.
Since the class of graphs with the *-C1P is a generalization of the classes of graphs whose adjacency or augmented adjacency matrix has the C1P, and since those have path eccentricity at most 1, we conjecture that this new class also has path eccentricity at most 1:

\begin{restatable}{conj}{last}
  \label{thm:mC1P_pe1}
     If a graph $G$ has the *-C1P then $pe(G)\leq 1$.  
\end{restatable}

\paragraph{Paper organization} In \cref{sec:preli} we set up our notations, definitions and introduce all new notions. In \cref{sec:k-AT-free} we prove our general result on $k$-AT-free graphs. In \cref{sec:fullmixed} we prove that graphs having the *-C1P are 2-AT-free.


\section{Preliminaries}\label{sec:preli}

In this paper, every graph $G=(V(G),E(G))$ (or $(V,E)$ if the context is clear) is finite, simple and undirected with $V$ the set of vertices of $G$ and $E$ the set of edges of $G$. Observe that if a graph is not connected, its path eccentricity is unbounded. Therefore we will study only connected graphs in this paper.
\bigskip

Given two graphs $G$ and $H$, the graph $H$ is a \emph{subgraph} of $G$ if $V(H)\subseteq V(G)$ and $E(H)\subseteq E(G)$.
$H$ is an \emph{induced subgraph} of $G$ if $V(H)\subseteq V(G)$ and two vertices of $H$ are adjacent in $H$ if and only if they are adjacent in $G$.

A path $P$ is a graph with vertex set $V(P)=\{v_0, \ldots, v_\ell\}$ for some $\ell \geq 0$, such that $v_i$ is adjacent to $v_{i+1}$, for every $0\leq i < \ell$. We call $\ell$ its \emph{length} and $v_0,v_\ell$ its \emph{extremities}. 
A \emph{cycle of length $\ell\geq 3$}, denoted $C_\ell$, consists in a path of length $\ell-1$ with an edge between its extremities.
For two vertices $u,v$ in $V(G)$, we say that $G$ has an (induced) path of length $\ell$ between $u$ and $v$ if $G$ admits an (induced) subgraph isomorphic to a path of length $\ell$ with extremities $u,v$.

For convenience and when there is no ambiguity, we will not distinguish between a set of vertices and the subgraph it induces.
For example, for a vertex $v$ in a path $P$, we write $P\setminus v$ to denote the subgraph of $P$ induced by $V(P)\setminus \{v\}$.

\bigskip
The \emph{distance} between two vertices $u$ and $v$ in $G$, denoted by $d(u,v)$, is the minimum length of a path between $u$ and $v$.
Given a subgraph $H$ of $G$ and a vertex $u$ in $G$, the distance between $u$ and $H$ (denoted $d(u,H)$) is the minimum distance between $u$ and a vertex of $H$, thus $d(u,H)=\min\limits_{v\in V(H)} d(u,v)$.

The \emph{eccentricity} of a path $P$ of $G$ is the maximal distance between $P$ and vertices of $G$, and is noted $ecc(P)\coloneqq \max\{d(u,P):u\in V\}$. The  \emph{path eccentricity} of $G$ is defined as the minimum of the eccentricity over all paths of $G$, i.e. $pe(G)\coloneqq \min\{ecc(P):P \text{ path of }G \}$. 

\bigskip
Given a set $S$ of vertices of $G$, we denote by $N[S]$ the closed neighborhood of $S$, which is the union of $S$ and all the vertices adjacent to a vertex in $S$. 
The open neighborhood of $S$, denoted $N(S)$, corresponds to $N[S]\setminus S$.
More generally, for $k\geq 1$, the neighborhood at distance $k$ of $S$, denoted $N^k[S]$, is the set of all vertices at distance at most $k$ from $S$.
For convenience, if $S = \{v\}$ for some $v\in V(G)$, we write $N(v)$ (resp. $N[v]$ and $N^k[v]$) instead of $N(\{v\})$ (resp. $N[\{v\}]$ and $N^k[\{v\}]$).

A graph is \emph{bipartite} if its vertex set can be partitioned into two sets such that two vertices are adjacent only if they are not in the same set.

A \emph{claw}, also denoted $K_{1,3}$ in the literature, is a bipartite graph $(X\cup Y, E)$ with $\vert X\vert=1$, $\vert Y \vert=3$ and the maximum number of edges. For $k\geq 1$, a \emph{$k$-subdivided-claw} is a claw where each edge is replaced by a path of length $k$. In particular, a claw is a 1-subdivided-claw.

\subsection{On asteroidal triples}
This part introduces the notion of $k$-asteroidal triples which will be used in Section~\ref{sec:k-AT-free} and in Section~\ref{sec:mixed}.

An \emph{asteroidal triple} is a set of three vertices such that each pair of vertices is joined by a path that avoids the closed neighborhood of the third.
A graph $G$ is said to be \emph{asteroidal triple-free}, or \emph{AT-free} for short, if it does not contain any asteroidal triple, i.e. for every set of three of vertices of $G$ there is a pair of them such that every path between them intersects the closed neighborhood of the third.

A generalization of this class of graphs was introduced by Machada and de Figueirdo~\cite{kATfreeMachadoFigueiredo} in the context of the existence of a class of graphs with arbitrarily large cycles for which breadth first search would always return high-eccentricity vertices.
A \emph{$k$-asteroidal triple}, noted $k$-AT, is a set of three vertices such that each pair of vertices is joined by a path that avoids the neighborhood at distance $k$ of the third. A graph is \emph{$k$-AT-free} if it does not contain a $k$-AT. 
Note that $k$-subdivided-claws and cycles of size between $3k$ and $3k+2$ are examples of $k$-AT-free graphs admitting $(k-1)$-AT. Moreover, any graph containing a $k$-subdivided-claw or a cycle of size $3k$ as an induced isometric\footnote{$H$ is an \emph{isometric} subgraph of $G$ if the distances between the vertices of $H$ are preserved in $G$.} subgraph can not be $(k-1)$-AT-free, and is therefore $k_0$-AT-free for $k_0 \geq k$.

\subsection{On the consecutive ones property}
This part focuses on the notion of C1P and is not required for Section~\ref{sec:k-AT-free}.
We are interested in the following matrix property, that we will translate into a graph property.

\begin{defi*}
    A matrix is said to have the \emph{consecutive ones property} (for columns) if it exists a permutation of its rows that places the 1s consecutively in every column.
\end{defi*}

We define three sets of matrices associated to an $n$-vertex graph $G$, each matrix having $n$ rows and $n$ columns, the set of \emph{adjacency matrices}, the set of \emph{augmented adjacency matrices}\footnote{Also called \emph{neighborhood matrices}.}, 
and the set of \emph{partially augmented adjacency matrices} of $G$. Given an ordering $V(G) =\{u_1, u_2, \dots , u_n\}$ of the vertices, we define one matrix for each of the three set, these three matrices differ only on the diagonal. For each $1\leq i,j \leq n$, $i\neq j$, all three matrices have a 1 in row $i$ column $j$ if and only if $u_i, u_j$ are adjacent, 0 otherwise. On the diagonal, adjacency, augmented adjacency and partially augmented adjacency matrices contain respectively only 0s, only 1s, or arbitrary values in $\{0,1\}$.

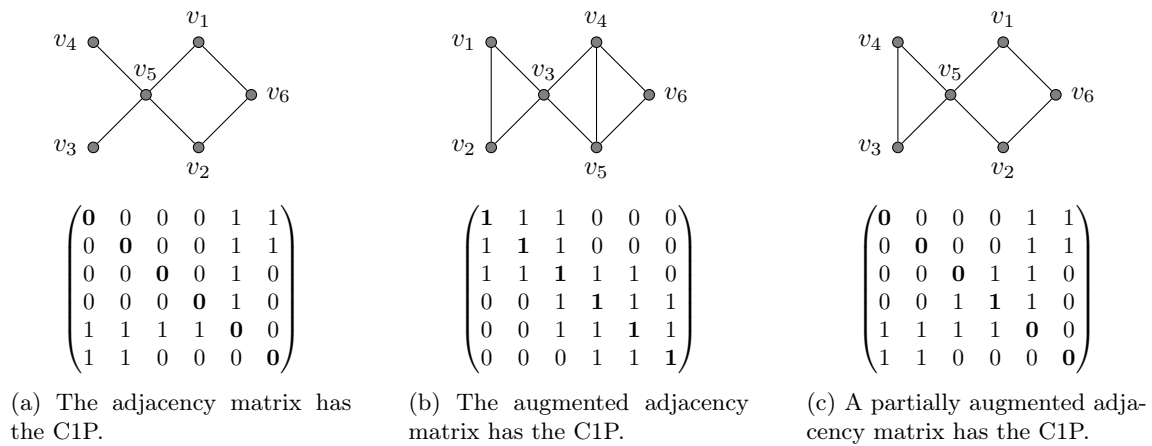
\begin{figure}[!htb]
\centering
\begin{subfigure}{.3\textwidth}
  \centering
    \begin{tikzpicture}[scale=0.7]
    \tikzstyle{vertex}=[circle, draw, fill=black!50,
                        inner sep=0pt, minimum width=4pt]
    \node[vertex] (v3) [label=left:$v_3$] at (0,0) {};
    \node[vertex] (v4) [label=left:$v_4$] at (0,2) {};
    \node[vertex] (v5) [label=above:$v_5$] at (1,1) {};
    \node[vertex] (v1) [label=above:$v_1$] at (2,2) {};
    \node[vertex] (v6) [label=right:$v_6$] at (3,1) {};
    \node[vertex] (v2) [label=below:$v_2$] at (2,0) {};
    \draw (v4) -- (v5) -- (v1) -- (v6) -- (v2) -- (v5) -- (v3);
\end{tikzpicture}
  \label{fig:sub11}
\end{subfigure}%
\hfill
\begin{subfigure}{.3\textwidth}
  \centering
    \begin{tikzpicture}[scale=0.7]
    \tikzstyle{vertex}=[circle, draw, fill=black!50,
                        inner sep=0pt, minimum width=4pt]
    \node[vertex] (v2) [label=left:$v_2$] at (0,0) {};
    \node[vertex] (v1) [label=left:$v_1$] at (0,2) {};
    \node[vertex] (v3) [label=above:$v_3$] at (1,1) {};
    \node[vertex] (v4) [label=above:$v_4$] at (2,2) {};
    \node[vertex] (v6) [label=right:$v_6$] at (3,1) {};
    \node[vertex] (v5) [label=below:$v_5$] at (2,0) {};
    \draw (v3)--(v1) -- (v2) -- (v3) -- (v4) -- (v5) -- (v6) -- (v4); 
    \draw (v3) -- (v5); 
\end{tikzpicture}
  \label{fig:sub21}
\end{subfigure}\hfill
\begin{subfigure}{.3\textwidth}
  \centering
    \begin{tikzpicture}[scale=0.7]
    \tikzstyle{vertex}=[circle, draw, fill=black!50,
                        inner sep=0pt, minimum width=4pt]
    \node[vertex] (v3) [label=left:$v_3$] at (0,0) {};
    \node[vertex] (v4) [label=left:$v_4$] at (0,2) {};
    \node[vertex] (v5) [label=above:$v_5$] at (1,1) {};
    \node[vertex] (v1) [label=above:$v_1$] at (2,2) {};
    \node[vertex] (v6) [label=right:$v_6$] at (3,1) {};
    \node[vertex] (v2) [label=below:$v_2$] at (2,0) {};
    \draw (v4) -- (v5) -- (v1) -- (v6) -- (v2) -- (v5) -- (v3) -- (v4);
\end{tikzpicture}
  \label{fig:sub1}
\end{subfigure}

\medskip
  \begin{subfigure}{.3\textwidth}
  \centering
  \resizebox{.7\textwidth}{!}{$
  \begin{pmatrix}
     \bf{0} & 0 & 0 & 0 & 1 & 1\\
     0 & \bf{0} & 0 & 0 & 1 & 1\\
     0 & 0 & \bf{0} & 0 & 1 & 0\\
     0 & 0 & 0 & \bf{0} & 1 & 0\\
     1 & 1 & 1 & 1 & \bf{0}  & 0\\
     1 & 1 & 0 & 0 & 0 & \bf{0} 
  \end{pmatrix}
  $}
  \caption{The adjacency matrix has the C1P.}
  \label{fig:sub_adjacency}
\end{subfigure}\hfill
\begin{subfigure}{.3\textwidth}
  \centering
  \resizebox{.7\textwidth}{!}{$
  \begin{pmatrix}
     \bf{1} & 1 & 1 & 0 & 0 & 0\\
     1 & \bf{1} & 1 & 0 & 0 & 0\\
     1 & 1 & \bf{1} & 1 & 1 & 0\\
     0 & 0 & 1 & \bf{1} & 1 & 1\\
     0 & 0 & 1 & 1 & \bf{1}  & 1\\
     0 & 0 & 0 & 1 & 1 & \bf{1} 
  \end{pmatrix}
  $}
  \caption{The augmented adjacency matrix has the C1P.}
  \label{fig:sub_augmented}
\end{subfigure}\hfill
\begin{subfigure}{.3\textwidth}
  \centering
  \resizebox{.7\textwidth}{!}{$
  \begin{pmatrix}
     \bf{0} & 0 & 0 & 0 & 1 & 1\\
     0 & \bf{0} & 0 & 0 & 1 & 1\\
     0 & 0 & \bf{0} & 1 & 1 & 0\\
     0 & 0 & 1 & \bf{1} & 1 & 0\\
     1 & 1 & 1 & 1 & \bf{0}  & 0\\
     1 & 1 & 0 & 0 & 0 & \bf{0} 
  \end{pmatrix}
  $}
  \caption{A partially augmented adjacency matrix has the C1P.}
  \label{fig:sub_partially}
\end{subfigure}
\caption{Examples of graphs and the associated matrices with ordering function $\mu(v_i)=i, \forall i$.}
\label{fig:examplesC1P}
\end{figure}


\medskip

An \emph{ordering function} of the vertices of a graph $G$ is an injective function $\mu:V(G) \rightarrow \ZZ$. It gives a total order on $V(G)$, denoted by $(\cleq_{\mu},V)$, such that $u\cleq_{\mu} v$ iff $\mu(u)\leq \mu(v)$. 
The interval $[u,v]_{\cleq_{\mu}}$ denotes the set of preimage of the integers in $[\mu(x), \mu(v)]$ and $N_{\cleq_{\mu}}(v)$ denotes the image of $N(v)$ by $\mu$. When the ordering function used is clear or irrelevant, we will simplify the notation $\cleq_{\mu}$ by $\cleq$. 

\medskip

Observe that, given a total ordering of $V(G)$ there are a unique adjacency and a unique augmented adjacency, and reciprocally, both these matrices induce a natural total order on the vertices of $G$, the first vertex for the order is the one corresponding to the first row, the second in the order correspond to the second row and so on. 

Therefore, we can equivalently define a graph whose adjacency matrix (resp. augmented adjacency matrix) has the consecutive ones property as a graph for which there exists an ordering function of its vertices sending the open (resp. closed) neighborhood of every vertex to a consecutive set.
For examples of such graphs, see \cref{fig:sub_adjacency} and \cref{fig:sub_augmented}.

\medskip

Similarly, a partially augmented adjacency matrix of $G$ corresponds to a unique total order, and a total order corresponds to partially augmented adjacency matrices that differ only on the diagonal. 

We say that a graph $G$ has the \emph{partially augmented consecutive ones property}, noted *-C1P, if there exists a partially augmented adjacency matrix of $G$ which has the C1P. In other words, a graph has the *-C1P if there exists an ordering of its vertices sending the open or closed neighborhood of every vertex to a consecutive set. See \cref{fig:sub_partially} for an example of graph with the *-C1Pbut for which there is no ordering such that its (augmented) adjacency matrix has the C1P.

Observe that if a graph $G$ admits an adjacency matrix (resp. augmented adjacency matrix) which has the C1P, then $G$ has the *-C1P.
However the other direction is not true, as we saw with \cref{fig:sub_partially}.
There exists families of graphs, such as the ladder on $2k$ vertices with the last rung being a $K_4$ (see \cref{fig:ladderK4}), for which it is needed to partially augment the adjacency matrix for it to have the C1P. 

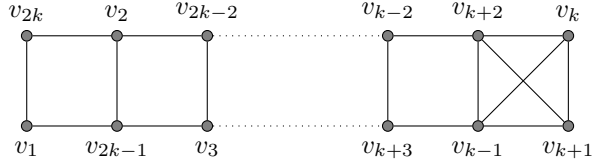
\begin{figure}[!htb]
\centering
    \begin{tikzpicture}[scale=0.6]
    \tikzstyle{vertex}=[circle, draw, fill=black!50,
                        inner sep=0pt, minimum width=4pt]
    \node[vertex] (v1) [label=below:$v_1$] at (0,0) {};
    \node[vertex] (v2k) [label=above:$v_{2k}$] at (0,2) {};
    \node[vertex] (v2k_1) [label=below:$v_{2k-1}$] at (2,0) {};
    \node[vertex] (v2) [label=above:$v_{2}$] at (2,2) {};
    \node[vertex] (v3) [label=below:$v_3$] at (4,0) {};
    \node[vertex] (v2k_2) [label=above:$v_{2k-2}$] at (4,2) {};

    \node[vertex] (vk_3) [label=below:$v_{k+3}$] at (8,0) {};
    \node[vertex] (vk_2) [label=above:$v_{k-2}$] at (8,2) {};
    \node[vertex] (vk_1) [label=below:$v_{k-1}$] at (10,0) {};
    \node[vertex] (vk2) [label=above:$v_{k+2}$] at (10,2) {};
    \node[vertex] (vk1) [label=below:$v_{k+1}$] at (12,0) {};
    \node[vertex] (vk) [label=above:$v_{k}$] at (12,2) {};
    
    \draw (v2) -- (v2k) -- (v1) -- (v2k_1) -- (v2) -- (v2k_2) -- (v3) -- (v2k_1);
    \draw (vk2) -- (vk) -- (vk1) -- (vk_1) -- (vk2) -- (vk_2) -- (vk_3) -- (vk_1) -- (vk);
    \draw (vk1)--(vk2);
    \draw[dotted] (v2k_2) -- (vk_2);
    \draw[dotted] (v3) -- (vk_3);
\end{tikzpicture}
\caption{Family of graphs having the *-C1P, with ordering function $\mu(v_i)=i, \forall i$.}
\label{fig:ladderK4}
\end{figure}

\section{Path eccentricity of $k$-AT-free graphs}\label{sec:k-AT-free}

The first class of graphs in which Cockayne et al.~\cite{linearAlgoPC_Cockayne} and Slater~\cite{LocatingCP} looked for a central path were the trees.
Recall that the \emph{$k$-subdivided-claw} is a tree consisting of three disjoint paths on $k$ vertices, and one vertex (its \emph{center}) adjacent to one extremity of each path. 
One can quickly observe that a $k$-subdivided-claw has path eccentricity $k$, hence trees have unbounded path eccentricity. Moreover, in order to bound the path eccentricity of an hereditary graph class by $k$, it is necessary that the class avoids the $k$-subdivided-claws as an induced subgraph. 

Corneil, Olariu and Stewart~\cite{ATfree} studied a now well known class of graphs, that happens to avoid 2-subdivided-claws: the asteroidal triple-free graphs.
They showed in particular that in every AT-free graph, there is a path with eccentricity at most 1. 



In this section, we generalize this result to the $k$-AT-free graphs.
Recall that a $k$-AT-free graph is a graph such that for every set of three vertices there is a pair of them such that every path between them intersects the neighborhood at distance $k$ of the third.


We prove that those graphs have path eccentricity at most $k$. The proof works towards a contradiction using a minimal counter-example. Suppose the graph is $k$-AT-free and have path eccentricity at least $k+1$, we then choose a shortest path of minimal length among those maximizing their neighborhood at distance $k$, and either find a $k$-AT or show that this path can be extended into a path with a strictly larger neighborhood at distance $k$. 

\kATfree*

\begin{proof}

    Let $G$ be a $k$-AT-free graph. Arguing by contradiction, suppose that $G$ has path eccentricity at least $k+1$. Let $P$ be a path that maximize $\vert N^{k}[P]\vert$, and among those, one of minimum length. Let $u,v$ be its extremities. Since $pe(G)\geq k+1$, there is a vertex $w$ of $G$ such that $w\notin N^{k}[P]$. 
    
    Observe that if we can find a path $P'$ such that $V(P)\subseteq V(P')$ and $w\in N^{k}[P']$, then $|N^{k}[P']| > |N^{k}[P]|$, which contradicts the maximality of $P$. Let us call such a path $P'$ an \emph{improving} path: finding an improving path in $G$ implies a contradiction.

    Let $a$ be a vertex of $P$ such that $d(w,a)=d(w,P)$, and let $P_{wa}$ be a shortest path from $w$ to $a$. Note that by definition of $a$, $P$ and $P_{wa}$ intersect only in $a$. Let $P_{ua}$ (respectively $P_{va}$) be the restriction of $P$ between $a$ and $u$ (respectively $v$). 
    
    Observe that $a\neq u$, otherwise $P_{wa}\cup P$ would be an improving path. For the same reason, $a\neq v$, and thus $P$ has length at least $2$.

    \begin{claim}\label{cl:u'v'}
        There is a vertex $u'$ at distance $k$ from $u$ such that $u'$ does not belong to $N^{k}[(P\setminus u) \cup P_{wa}]$; and similarly, 
        there is a vertex $v'$ at distance $k$ from $v$ that does not belong to $N^{k}[(P\setminus v) \cup P_{wa}]$.
    \end{claim}

    \begin{proofclaim}
        Notice first that there has to be a vertex $u'$ that belongs to $N^{k}[P]$ and does not belong to $N^{k}[P\setminus u]$, otherwise $N^{k}[P] \subseteq N^{k}[P\setminus u]$ and $P\setminus u$ is a strictly shorter path than $P$ with the same neighborhood, which contradicts the definition of $P$. So $u'$ has to be at distance $k$ of $u$.

        Suppose now that $u'$ belongs to $N^{k}[P_{wa}]$.  Let $y$ be a vertex of $P_{wa}$ closest to $u'$. Then $P_{yu'}$, the path realizing this shortest distance, has length at most $k$. Note that, by definition of $y$, $P_{yu'}$ intersects $P_{wa}$ only in $y$.
        Observe also that $P_{yu'}$ does not intersect $P\setminus u$, otherwise there would be a path of length at most $k$ between $u'$ and $P\setminus u$. 
        Let $P_{u'u}$ be a shortest path from $u'$ to $u$, and let $y'$ be the vertex in the intersection between $P_{yu'}$ and $P_{u'u}$ that is the closest from $u$ along $P_{u'u}$. Let $P_{wy'u}$ be the path resulting from the union of the restriction of $P_{wa}$ between $w$ and $y$, the restriction of $P_{yu'}$ between $y$ and $y'$, and the restriction of $P_{u'u}$ between $y'$ and $u$. This path is internally disjoint from $P$. Thus the path resulting from the union of $P_{wy'u}$ and $P$ is an improving path, a contradiction. Therefore, $u'\notin N^{k}[P_{wa}]$.
        
        By symmetry of the roles of $u$ and $v$, we prove the existence of $v'$ similarly.
    \end{proofclaim}

    Let us now show that $\{u',v',w\}$ is a $k$-AT, \textit{i.e.} for each pair of vertices in the set, there is a path between this pair avoiding the neighborhood at distance $k$ of the third vertex of the set.

    Let $P_{u'u}$ be a shortest path from $u'$ to $u$, let $P_{v'v}$ be a shortest path from $v'$ to $v$, and let $P_{u'v'}=P \cup P_{u'u} \cup P_{v'v}$, $P_{wu'}=P_{wa} \cup P_{ua} \cup P_{u'u}$ and $P_{wv'}= P_{wa} \cup P_{va} \cup P_{v'v}$.

    Suppose that $w\in N^{k}[P_{u'v'}]$, then $P_{u'v'}$ is an improving path since $V(P)\subseteq V(P_{u'v'})$, which is impossible. Thus $P_{u'v'}$ is a path that avoids $N^{k}[w]$.

    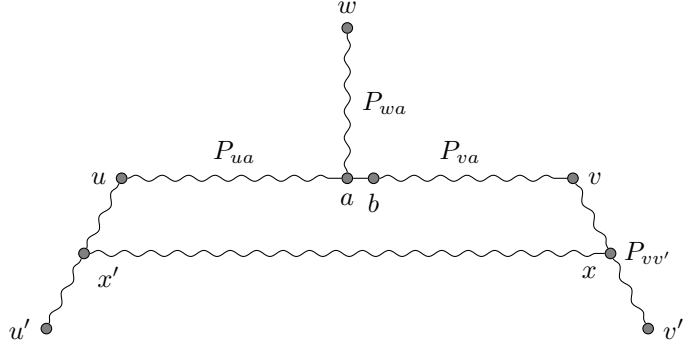
\begin{figure}[h]
    \centering
    \begin{tikzpicture}
    \tikzstyle{vertex}=[circle, draw, fill=black!50,
                        inner sep=0pt, minimum width=4pt]
    \node[vertex] (u) [label=left:$u$] at (0,0) {};
    \node[vertex] (v) [label=right:$v$] at (6,0) {};
    \node[vertex] (a) [label=below:$a$] at (3,0) {};
    \node[vertex] (b) [label=below:$b$] at (3.35,0) {};
    \node[vertex] (w) [label=above:$w$] at (3,2) {};
    \node[vertex] (u') [label=left:$u'$] at (-1,-2) {};
    \node[vertex] (v') [label=right:$v'$] at (7,-2) {};
    \node[vertex] (x) [label=below left:$x$] at (6.5,-1) {};
    \node[vertex] (x') [label=below right:$x'$] at (-0.5,-1) {};
    \draw (a) -- (b);
    \draw  [snake=snake, segment amplitude=.4mm](b) --  (v) -- (x) -- (v');
    \draw  [snake=snake, segment amplitude=.4mm] (x') -- (x);
    \draw [snake=snake, segment amplitude=.4mm] (u') -- (x') -- (u) -- (a) -- (w);
    \node at (3.5,1) {$P_{wa}$};
    \node at (1.5,0.35) {$P_{ua}$};
    \node at (4.5,0.35) {$P_{va}$};
    \node at (7,-1) {$P_{vv'}$};

\end{tikzpicture}
    \caption{Representation of the structure used in the proof of \cref{thm:kATfree_pek}.}
    \label{fig:kATconstruction}
    \end{figure}

    Suppose now that $u' \in N^{k}[P_{wv'}]$. Recall that $P_{wv'} = P_{wa} \cup P_{va} \cup P_{v'v}$ and by Claim \ref{cl:u'v'}, $u' \notin N^{k}[(P_{ua}\setminus u) \cup P_{va} \cup P_{wa}]$, it follows that $u'$ is in $N^{k}[P_{v'v}]$. Let $x$ be a vertex of $P_{v'v}$ closest to $u'$. Then $P_{xu'}$, a shortest path between $x$ and $u'$, has length at most $k$. Also, by definition of $x$, $P_{xu'}$ intersects $P_{v'v}$ only in $x$. 
    Observe that $P_{xu'}$ does not intersect $P\setminus u$, since by Claim \ref{cl:u'v'} $u'$ is at distance at least $k+1$ from $P\setminus u$. Let $x'$ be the vertex in the intersection of $P_{u'u}$ and $P_{xu'}$, closest to $x$ along $P_{xu'}$. A representation of the structure is depicted in Figure \ref{fig:kATconstruction}.
    Let $P_{vu}$ be the union of the restriction of $P_{v'v}$ between $v$ and $x$, the restriction of $P_{xu'}$ between $x$ and $x'$ and the restriction of $P_{u'u}$ between $x'$ and $u$. From the previous observation, $P_{vu}$ is internally disjoint from $P$. 
    Finally, let $b$ be the vertex next to $a$ on the path $P_{va}$, and let $P_{bv}$ be the restriction of $P$ between $b$ and $v$ (potentially $b=v$ and $P_{bv}$ is a trivial path). Then $P'= P_{bv} \cup P_{vu} \cup P_{ua} \cup P_{wa}$ is a path with $V(P) \subseteq V(P')$ and $w \in V(P')$, so $P'$ is an improving path, which is impossible. So $P_{wv'}$ is a path at distance at least $k+1$ from $u'$.

    By symmetry of the roles of $u',v'$, we can prove similarly that $P_{wu'}$ avoids $N^{k}[v']$.
    Thus $\{u',v',w\}$ forms a $k$-AT, a contradiction.
\end{proof}

We note that this bound is tight as $k$-subdivided-claws are $k$-asteroidal triple-free.

This result gives us a powerful way to bound the path eccentricity of classes of graphs through their structure. It is applied in Section~\ref{sec:mixed} to show that graphs having the *-C1P have path eccentricity bounded by 2.


\section{Graphs having the partially augmented consecutive ones property}\label{sec:fullmixed}
In this section, we focus on the graphs having the *-C1P.  Having the C1P on the adjacency or augmented adjacency matrix ends up being quite restrictive on the structure of the graphs. Indeed, the first class corresponds exactly to the biconvex graphs~\cite{ChenBiconv} while the second corresponds exactly to the unit interval graphs~\cite{GardiInterval}, and they both have path eccentricity at most 1. Observe that the latter class is a subset of the AT-free graphs, which is not the case of the former, as shown in Figure~\ref{fig:biconvexAT}. In fact, they are quite different as the intersection of both properties, that is, the class of biconvex graphs that also admit a unit interval representation is very restrictied. It is exactly the class of path forests. Our interest lies in considering a larger class of graphs, we study a natural generalization of both of these classes, the *-C1P.

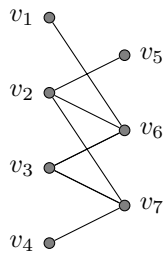
\begin{figure}[H]
    \centering
    \tikzstyle{vertex}=[circle, draw, fill=black!50,
                        inner sep=0pt, minimum width=4pt]
    \begin{tikzpicture}[scale = 0.25]
            \draw (0,12) node[vertex] (A1) [label=left:$v_1$] {}
            (0,8) node[vertex] (A2) [label=left:$v_2$] {}
            (0,4) node[vertex] (A3) [label=left:$v_3$] {}
            (0,0) node[vertex] (A4) [label=left:$v_4$] {}
            (4,10) node[vertex] (A5) [label=right:$v_5$] {}
            (4, 6) node[vertex] (A6) [label=right:$v_6$] {}
            (4, 2) node[vertex] (A7) [label=right:$v_7$] {};
            \draw (A1) -- (A6) -- (A3) -- (A7) -- (A2) -- (A5);
            \draw (A2)--(A6) -- (A3)--(A7)--(A4);
        \end{tikzpicture}
    \caption{A biconvex graph with the asteroidal triple $\{v_1,v_4,v_5\}$.}
    \label{fig:biconvexAT}
\end{figure}

In this section we first give some useful lemmas on the order of the vertices in the neighborhood of a path in a graph having the *-C1P.
They are then used to prove our main theorem of the section, that if a graph has the *-C1P, then it cannot have a 2-AT. It follows then, by Theorem~\ref{thm:kATfree_pek}, that their path eccentricity is at most 2. 


\subsection{Induced paths and *-C1P}\label{sec:techniC1P}


We will show that the consecutiveness on the neighborhoods implies a strong condition on the order of vertices in induced paths of the graphs, and thus also on induced cycles.
This lemma is a key result to understand the structure of graphs defined with the *-C1P.

\begin{lemma}\label{lem:order_Puv}
    Let $G$ be a graph with the *-C1P with the related ordering function $\mu$, and let $P_{uv}=z_0z_1z_2\dots z_{\ell-1}z_{\ell}$ be an induced path from $u=z_0$ to $v=z_\ell$ of length $\ell$ in $G$. Then we have, with $i \in \NN,$ 
    \begin{itemize}
        \item $(u, z_2, z_4,\dots)$ and $(v, z_{\ell -2}, z_{\ell -4}, \dots)$ are monotonic sequences according to $\cleq$;
        \item $[u,z_{2\lfloor \ell/2 \rfloor}]_{\cleq}\subseteq N_{\cleq}[V(P_{uv})]$ and $[z_{\ell -2\lceil \ell/2 \rceil},v]_{\cleq}\subseteq N_{\cleq}[V(P_{uv})]$;
        \item if $P_{uv}$ is of even length, then $[u,v]_{\cleq}\subseteq N_{\cleq}[V(P_{uv})]$.
    \end{itemize}
\end{lemma}

\begin{proof}
    Let us first prove by induction on $0\leq i \leq \frac{\ell}{2}$ that $[u, z_{2i}]_{\cleq} \subseteq N_{\cleq}[V(P_{uz_{2i}})]$, where $P_{uz_{2i}}$ is the restriction of $P_{uv}$ between $u$ to $z_{2i}$, and $(u, z_2, \dots , z_{2i})$ is a monotonic sequence according to $\cleq$.
    
    The statement is trivial for $i=0$. For $i=1$, the *-C1P of $G$ and $\{u,z_2\} \in N[z_1]$ give that for all $x$ with $z_0\cleq x \cleq z_2$, $x\in N[z_1]$ so $[u, z_2]_{\cleq} \subseteq N_{\cleq}[V(P_{uz_2})]$. On the other hand, the ordering function due to the *-C1P of $G$ gives a total order on $V(G)$. Therefore $u$ and $z_2$ are comparable and the sequence $(u, z_2)$ is monotonic according to $\cleq$.
    
    Suppose now that the induction hypothesis is true up to $i$ (with $i<\frac{\ell}{2}-1$). The *-C1P of $G$ and $\{z_{2i},z_{2i+2}\} \in N[z_{2i+1}]$ give that $[z_{2i},z_{2i+2}]_{\cleq}\subseteq N_{\cleq}[z_{2i+1}]$. 
    Combined with the first part of the induction hypothesis, we have $[u,z_{2i}]_{\cleq}\cup [z_{2i},z_{2i+2}]_{\cleq} \subseteq N_{\cleq}[V(P_{uz_{2i}})]\cup N_{\cleq}[z_{2i+1}] \subseteq N_{\cleq}[V(P_{uz_{2i+2}})]$. 
    
    The second part of the induction hypothesis gives us that $(u, z_2, \dots, z_{2i})$ is monotonic according to $\cleq$. The sequence $(u, z_2, \dots, z_{2i}, z_{2i+2})$ has to be monotonic too.
    Indeed, if it was not, then either $z_{2i+2}$ is inserted in the sequence $(u, z_2,\dots, z_{2i})$, so $\mu(z_{2i+2})\in [u,z_{2i}]_{\cleq}\subseteq N_{\cleq}[P_{uz_{2i}}]$ and then $P_{uv}$ is not induced, or $z_{2i+2}$ is such that $\mu(z_{2i+2}) \leq \mu(u)$ in particular we would have that  $\mu(u) \in [z_{2i+2},z_{2i}]_\cleq \subseteq  N_\cleq[z_{2i+1}]$ which would imply that $P_{uv}$ is again not induced.

    \medskip

    By symmetry of the roles of $u$ and $v$, we obtain a similar result on the alternating sequence starting from $v$, which completes the proof for the first two statements of the Lemma.

    
    At last, if $\ell$ is even, then $u$ and $v$ are in the same sequence and we then have $[u,v]_{\cleq}\subseteq N_{\cleq}[V(P_{uv})]$. 
\end{proof}

This lemma enables us to deduce a nice property on the cycles in the graphs with the *-C1P.

To demonstrate the usefulness of the above lemma, we provide a brief proof of the following proposition.

\begin{prop}\label{lem:mC1P_noC5}
    If a graph $G$ has the *-C1P then it is $C_{\geq 5}$-free.
\end{prop}

\begin{proof} \sloppy
    Let $G$ be a graph having the *-C1P and $\mu$ the related ordering function. Let $C=v_0v_1\dots v_{\ell-1}$ be an induced cycle of length $\ell\geq 3$. 
    By applying Lemma~\ref{lem:order_Puv} on the subpaths of length $\ell-1$ of $C$, we get that for every $i_0<\ell$ and every $0\leq j<\ell/2$, the sequence $(v_{i_0},v_{(i_0+2 \mod \ell)},\dots v_{(i_0+2j \mod \ell)})$ is monotonic. Moreover, if $\ell\geq 6$, this sequences have size at least 3, thus we can combined such sequences if they have at least two vertices in common.

    If $\ell$ is even and $\ell\geq 6$, then $(v_0, v_2, \dots, v_{\ell-2})$, and $(v_{\ell-2},v_0, v_2, \dots, v_{\ell-4})$ are monotonic, thus $(v_{\ell-2},v_0, v_2, \dots, v_{\ell-2})$ is a monotonic sequence, thus $v_{\ell-2} \cleq v_0 \cleq v_{\ell-2}$, a contradiction.

    If $\ell$ is odd and $\ell\geq 6$, then the following monotonic sequences have all size at least 3: $(v_0, v_2, \dots, v_{\ell-3})$, $(v_2, \dots, v_{\ell-3}, v_{\ell-1})$, $(v_4, \dots, v_{\ell-1}, v_1)$,  $(v_{\ell-1},v_1, v_3, \dots, v_{\ell-4})$, $(v_1, v_3, \dots, v_{\ell-2})$, and $(v_3, \dots, v_{\ell-2}, v_0)$. We can thus combined them and obtain the following monotonic sequence: $(v_0, v_2,\dots, v_{\ell-1} , v_1 ,\dots , v_{\ell-2} , v_0)$, and $v_0 \cleq v_1 \cleq v_0 $, a contradiction.

    It is easy to check by an exhaustive search that it is impossible to find an order on the vertices respecting the *-C1P if $\ell=5$. Thus $\ell \leq 4$.    
\end{proof}

This emphasizes that the AT-free graphs are not a subclass of graphs having the *-C1P. Indeed, a $C_5$ is $AT$-free but it has not the *-C1P. 

Using Lemma~\ref{lem:order_Puv}, we can prove the following lemma on the ordering of the neighborhood of the extremities of an induced path and a vertex avoiding it. 
For any set of vertices $S$, we note $\min_{\cleq} (S)$ (respectively $\max_{\cleq} (S)$) the smallest (respectively largest) element of $S$ according to $\mu$.

\begin{lemma}\label{lem:path_neighborhood}
    Let $P_{uv}$ be an induced path of odd length between $u$ and $v$ of a graph $G$ with the *-C1P. Let $x$ be a vertex of $G$ such that $x$ does not belong to $N[P_{uv}]$. Then:
    \begin{itemize}
        \item if $v \cleq u \cleq x$ then $\max_{\cleq}(N(u))\cleq x$ and $\max_{\cleq}(N(v))\cleq x$;
        \item if $x\cleq u \cleq v$ then $x \cleq \min_{\cleq}(N(u))$ and $x \cleq \min_{\cleq}(N(v))$;
        \item if $u \cleq x \cleq v$ then $\max_{\cleq}(N(v))\cleq x \cleq \min_{\cleq}(N(u))$.
    \end{itemize}
\end{lemma}
    
\begin{proof}
    Let $G$, $P_{uv}$ and $x$ be as defined in the lemma.

    Let $w_u$ and $w_v$ be the neighbor of respectively $u$ and $v$ on $P_{uv}$.
    Note that $P_{uw_v}$, the restriction of $P_{uv}$ between $u$ and $w_v$ is an induced path of even length so by Lemma~\ref{lem:order_Puv}, $[u,w_v]_\cleq \subseteq N_\cleq[P_{uw_v}]$. Since $x$ does not belong to $N_\cleq[P_{uw_v}]$, $x$ is either larger or smaller than all the vertices in $[u,w_v]_\cleq$.

    If $u \cleq x$, then $x$ is larger than all the vertices in $[u,w_v]_\cleq$, in particular $x$ is larger that $w_v\in N(v)$. Since $G$ has the *-C1P, either $N(v)$ or $N[v]$ is consecutive, and $x$ is larger than all the vertices of $N(v)$, in particular $\max_{\cleq}(N(v))$.

    Similarly, if $x \cleq u$, then $x$ is smaller than $w_v\in N(v)$ and by the *-C1P of $G$, $x$ is smaller than all the vertices of $N(v)$, in particular $\min_{\cleq}(N(v))$.

    Finally, with the symmetric analysis with $P_{w_uv}$, the restriction of $P_{uv}$ between $w_u$ and $v$, we obtain the relations with $\max_{\cleq}(N(u))$ and $\min_{\cleq}(N(u))$.
\end{proof}

\subsection{Main result}\label{sec:mixed}


Using the previous lemmas, we show that if our graph $G$ with the *-C1P had a 2-AT, there would be only a few ways to order those three vertices and their respective neighborhood, each yielding a contradiction. Remember that the graph represented in \cref{fig:biconvexAT} has the *-C1P but also contains an asteroidal triple, therefore the following theorem is tight.

\twoATfree*

\begin{proof}
    Let $G$ be a graph with the *-C1P and let $\mu$ be the relative ordering function. Assume for the sake of contradiction that there exists a 2-AT, $\{a,b,c\}$, where $a\cleq b \cleq c$. Then there are three induced paths in $G$, $P_{ab}$ from $a$ to $b$ (resp. $P_{ac}$ from $a$ to $c$ and $P_{bc}$ from $b$ to $c$) avoiding $N^2[c]$ (resp. $N^2[b]$ and $N^2[a]$).  

    
    \begin{claim}\label{claim:SAT_config}
        The vertices $a,b,c$ and their neighborhoods respect this ordering:
        \[\textstyle \max_\cleq (N(c))\cleq a \cleq b\cleq c\cleq \min_{\cleq} (N(a)).\]
    \end{claim}
    
    \begin{proofclaim}

        Notice first that, since $a\cleq b \cleq c$, by Lemma~\ref{lem:order_Puv}, $P_{ac}$ has odd length (otherwise $b\in [a,c]_\cleq \subseteq N_\cleq[P_{ac}]$, contradicting the fact that $\{a,b,c\}$ is a 2-AT). Thus we can use Lemma~\ref{lem:path_neighborhood} on vertex $b$ that avoids the neighborhood of $P_{ac}$, which gives that $\max_{\cleq} (N(c)) \cleq b\cleq \min_{\cleq} (N(a))$.
        Let us show here that $c\cleq \min_{\cleq}(N(a))$, the case $\max_{\cleq} (N(c))\cleq a$ follows from a similar argument.
        
        For the sake of contradiction, assume that there is a neighbor of $a$ in $[b,c]_\cleq$. By the *-C1P of $G$, the currently known ordering is:

        \begin{equation}\label{eq1}
            \textstyle a \cleq b\cleq \min_{\cleq} (N(a)) \cleq \max_{\cleq} (N(a))\cleq c.
        \end{equation}

        Let $w_a$ be the neighbor of $a$ on $P_{ab}$. From (\ref{eq1}), we have $b\cleq w_a \cleq c$, and  $w_a \notin N[P_{bc}]$ since $\{a,b,c\}$ is a $2$-AT, thus by Lemma~\ref{lem:order_Puv}, $P_{bc}$ has odd length. Then by applying Lemma~\ref{lem:path_neighborhood}, we get that $w_a \cleq \min_{\cleq}(N(b))$, and by *-C1P $\max_\cleq (N(a)) \cleq \min_{\cleq}(N(b))$ as the neighborhoods of $a$ and $b$ can not intersect. 

        Let $P_{w_ab}$ be the restriction of $P_{ab}$ between $w_a$ and $b$. Note that either $P_{ab}$ or $P_{w_ab}$ has odd length, none of them has $c$ in their neighborhood, and from (\ref{eq1}), we have $a \cleq b \cleq w_a \cleq c$. Thus, by applying Lemma~\ref{lem:path_neighborhood} to $c$ and either $P_{ab}$ or $P_{w_ab}$, we obtain that $\max_{\cleq}(N(b))\cleq c$. 
        Therefore, the currently known ordering can be summarized as:
        
        \begin{equation}\label{eq2}
            \textstyle a \cleq b\cleq \min_{\cleq} (N(a)) \cleq \max_{\cleq} (N(a)) \cleq \min_{\cleq} (N(b)) \cleq \max_\cleq(N(b)) \cleq c.
        \end{equation}

        Finally, let $w_b$ be the neighbor of $b$ on $P_{ab}$. Since $\{a,b,c\}$ is a 2-AT, $w_b$ does not belong to $N[P_{ac}]$. We saw previously that $P_{ac}$ has odd length, and from (\ref{eq2}), $a \cleq w_b \cleq c$. Thus, by Lemma~\ref{lem:path_neighborhood}, $\max_{\cleq}(N(c))\cleq w_b \cleq \min_{\cleq}(N(a))$, which implies that $ \max_{\cleq} (N(b)) \cleq \min_{\cleq}(N(a)) \cleq \min_{\cleq} (N(b))$, a contradiction. 
    \end{proofclaim}

    Let us now try to insert $N_{\cleq}(b)$ in the ordering given by Claim~\ref{claim:SAT_config}.
    Let $w_a$ be the neighbor of $a$ on $P_{ab}$, and let $P_{w_ab}$ be the restriction of $P_{ab}$ between $w_a$ and $b$. 
    From Claim~\ref{claim:SAT_config}, $b \cleq c \cleq w_a$, thus by Lemma~\ref{lem:order_Puv}, $P_{w_ab}$ has odd length (otherwise $c\in [b,w_a]_\cleq \subseteq N_\cleq[P_{ab}]$, contradicting the fact that $\{a,b,c\}$ is a 2-AT). Then by applying Lemma~\ref{lem:path_neighborhood} to $c$ and the induced path $P_{w_ab}$, we get that $c\cleq \min_{\cleq}(N(b))$.   

    Similarly, let $w_c$ be the neighbor of $c$ on $P_{bc}$, and let $P_{bw_c}$ be the restriction of $P_{bc}$ between $b$ and $w_c$. By Claim~\ref{claim:SAT_config}, $w_c \cleq a \cleq b$, thus by Lemma~\ref{lem:order_Puv}, $P_{bw_c}$ has odd length, and then, by applying Lemma~\ref{lem:path_neighborhood} to $a$ and the induced path $P_{bw_c}$, we get that $\max_{\cleq}(N(b)) \cleq a$.  

    Therefore $\max_{\cleq}(N(b)) \cleq a \cleq \min_{\cleq}(N(b))$, a contradiction.
\end{proof}




\bibliography{C1Pa}
\bibliographystyle{alpha}

\end{document}